\documentclass[a4paper,11pt]{amsart}
\usepackage{amsfonts}
\usepackage{amsmath}
\usepackage{amssymb}
\usepackage{amsthm}
\usepackage{url}
\usepackage[english]{babel}
\usepackage{verbatim}

\usepackage{enumitem, hyperref}

\makeatletter
\def\namedlabel#1#2{\begingroup
    #2%
    \def\@currentlabel{#2}%
    \phantomsection\label{#1}\endgroup
}
\makeatother

\numberwithin{equation}{section}

\def\argmax{{\rm{arg}}\max}
\def\Vl{V^{lim}}

\def\half{\frac{1}{2}}

\def\G{{\mathcal G}}

\def\P{{\mathcal P}}

\def\tC{{\mathbf{C}}}
\def\C{\tC}

\def\L{\buildrel L^1\over \longrightarrow}
\def\nto{\buildrel {n\rightarrow \infty}\over \longrightarrow}

\def\calL{{\mathcal L}}
\def\F{{\mathcal F}}
\def\P{{\mathbb P}}

\def\R{{\mathbb R}}
\def\A{{\mathcal A}}
\def\E{{\mathbb E}}

\def\N{{\mathbb N}}
\def\D{{\mathbf D}}

\newcommand{\sgn}{\operatorname{sgn}}

\def\ed{ \stackrel{d}{=} }

\theoremstyle{plain}
\newtheorem{thm}{Theorem}
\newtheorem{lem}[thm]{Lemma}
\newtheorem{proposition}[thm]{Proposition}

\theoremstyle{remark}
\newtheorem{remark}{Remark}
\theoremstyle{definition}
\newtheorem{defin}{Definition}

\title[PIA in continuous time]{On the policy improvement algorithm in continuous time}
\author{Saul D.\ Jacka}
\address{Department of Statistics, University of Warwick, UK}
\email{s.d.jacka@warwick.ac.uk}

\author{Aleksandar Mijatovi\'c}
\address{Department of Mathematics, Imperial College London, UK}
\email{a.mijatovic@imperial.ac.uk}

\begin{document}

\begin{abstract}
We develop a general approach to the Policy Improvement Algorithm (PIA) for stochastic control problems for continuous-time processes. 
The main results assume only that the controls lie in a compact metric space and give general sufficient conditions for
the PIA to be well-defined and converge in continuous time (i.e. without time discretisation). It emerges that the natural context 
for the PIA in continuous time is weak stochastic control. We give examples of control problems demonstrating
the need for the weak formulation as well as diffusion-based classes of problems where the PIA in continuous time is applicable. 
\end{abstract}



\maketitle

\section{Introduction}
The \textbf{Policy Improvement Algorithm (PIA)} has played a central role in control and optimisation for 
over half a century, see e.g. Howard's monograph~\cite{Howard}. The PIA yields an intuitive constructive approach
to optimal control by generating a sequence of policies $(\pi_n)_{n\in\N}$ whose payoffs $(V^{\pi_n})_{n\in\N}$
are improved at every step. Put differently, the payoffs 
$(V^{\pi_n})_{n\in\N}$ form a sequence of functions on the state space converging monotonically 
to the value function of the problem (see e.g. Section~\ref{sec:PIA} below for the precise definition).
In the stochastic setting, the PIA is perhaps 
most widely applied in the theory of Markov decision processes, see e.g.~\cite{Hordi, Meyn, Santos} and the 
references therein. Most of the literature on the PIA makes assumptions either on the process (e.g. finite/countable
state space or discrete time) or on the set of available controls (e.g. a finite set~\cite{Doshi}). 
In contrast, the present paper presents an abstract approach to the PIA in continuous time, allowing for 
an uncountable set of controls.  

The main aim of this work is two-fold: (1) define a general weak formulation for optimal control problems in continuous
time, without restricting the set of available controls, and (2) develop an abstract framework 
for in this setting for the PIA to work. The latter task involves stating a general set of assumptions
(see~\ref{As:A1}--\ref{As:A8} in Section~\ref{sec:PIA} below), under which the sequence of policies
$(\pi_n)_{n\in\N}$
can be constructed, prove that the PIA yields an increasing sequence of payoffs $(V^{\pi_n})_{n\in\N}$
(see Theorem~\ref{thm:1} below),
which converges to the value function of the stochastic control problem 
(see Theorem~\ref{thm:2} below),
\text{and} 
prove that a subsequence of policies 
$(\pi_n)_{n\in\N}$
converges uniformly on compacts
to an optimal policy 
$\pi^*$
with 
the payoff 
$V^{\pi^*}$
equal to the value function
(see Theorem~\ref{thm:3} below). In particular, our results imply that 
under general assumptions~\ref{As:A1}--\ref{As:A8},
an optimal policy
$\pi^*$
exists. 

The present paper presents a unified language for stating and solving general stochastic control problems in continuous time, 
which can in particular be used to describe simultaneously our recent 
results on the PIA for diffusions over the infinite~\cite{JMS_1} and finite~\cite{JMS_2} time horizons. 
The key distinction between this work and~\cite{JMS_1, JMS_2} lies in the fact that here we 
assume that the payoff
$V^{\pi_n}$
is sufficiently regular
for every policy $\pi_n$
produced by the PIA, which appears to be necessary for the algorithm to converge.
In contrast, in~\cite{JMS_1} (resp.~\cite{JMS_2}) we prove that this assumption is satisfied in the context 
of control problems for continuous-time diffusion processes over an infinite (resp. finite) time horizon. 

The remainder of the paper is organised as follows: Section~\ref{sec:Setting_and_examples}
gives the general weak formulation of the control problem and presents examples demonstrating
the necessity of the weak formulation. Section~\ref{sec:PIA} describes the PIA and states our main
results. Section~\ref{sec:Exmaples_diffusions} presents examples of the PIA in the context of diffusion
processes, based on~\cite{JMS_1, JMS_2}. The proofs of the results are given in Section~\ref{sec:Proofs}.

\section{The general problem: setting and examples}
\label{sec:Setting_and_examples}
\subsection{Setting}
\label{subsec:Setting}
Consider the following weak formulation of a general
\textit{optimal control problem}.
Given 
continuous functions 
$f:S\times A\to\R_+$
and 
$g:S\to\R_+$,
find for each $x\in S$
\begin{equation}
\label{eq:Value_Function}
V(x):=\sup_{\Pi\in \A_x}\E\left[\int_0^\tau
f(X^\Pi_t,
\Pi_t)dt+g(X^\Pi_\tau)1_{(\tau<\infty)}\right]
\end{equation}
where 
\begin{enumerate}[label=(\arabic*)]
\item the control process
$\Pi$,
defined on some filtered probability space $(\Omega,(\F_t)_{t\in \R_+},\F,\P)$, 
takes values in a
 compact metric space $A$ and is $(\F_t)$-adapted.
The topological space $S$
is the state space of the controlled process 
and $D$ is a domain (i.e. an open and connected subset) in $S$, such that
$D=\cup_{n=1}^\infty K_n$, where $\{K_n\}$ are an increasing family of compact sets 
in $S$ with $K_n$ contained in the interior of $K_{n+1}$ for all $n\in\N$; 
\label{item:def_strategy}
\item for each $a\in A$ we assume that 
$X^a$ is a strong Markov process with state space $S$ and a given 
(martingale) infinitesimal generator $\calL^a$ and
domain $\D^a$. Furthermore, we assume that there
exists a nonempty subset $\C$ 
of
$\cap_{a\in A}\D^a$ 
with the property that the map
$(x,a)\mapsto \calL^a \phi(x)$ is jointly continuous
on
$ D\times A$ for each $\phi\in \C$;
\label{item_2:Setting}
\item $\A_x$ consists of all control processes
$\Pi$ such that there exists an 
$(\F_t)$-adapted, right-continuous $S$-valued process $X^\Pi$ satisfying
\begin{enumerate}[label=(\roman*)]
\item $X^\Pi_0=x$;
\item the law of $(X^\Pi,\Pi)$ is unique;
\item for each $\phi\in \C$, 
\begin{equation}\label{gen}
\phi(X^\Pi_{t\wedge \tau})-\int_0^{t\wedge \tau}\calL^{\Pi_s}\phi(X^\Pi_s) ds\quad\text{ is a martingale,}
\end{equation}
where the stopping time  $\tau$ is the first exit
time of 
$X^\Pi$ 
from $D$; 
\item defining $J$ by
$$
J(x,\Pi):=\int_0^\tau f(X^\Pi_t, \Pi_t)dt+g(X^\Pi_\tau)1_{(\tau<\infty)},
$$ 
we have
$$
\int_0^{t\wedge \tau} f(X^\Pi_s, \Pi_s)ds+g(X^\Pi_\tau)1_{(\tau<\infty)}\L J(x,\Pi)
\qquad\text{as $t\to\infty$.}
$$
\end{enumerate}
We refer to the elements of $\A_x$ as controls.
\end{enumerate}

\begin{remark}
\label{rem:controls}
The stochastic basis, 
i.e. the filtered probability space 
$(\Omega,(\F_t)_{t\in \R_+},\F,\P)$, 
in the definition of the control process $\Pi$
in~\ref{item:def_strategy} above
may depend on 
$\Pi$.
In particular, the expectation in~\eqref{eq:Value_Function} corresponds to the probability
measure 
$\P$
under which the control $\Pi$
is defined. In the weak formulation, we are not required to fix a filtered probability 
space in advance but instead allow the control, together with its corresponding controlled process, to be 
defined on distinct stochastic bases for different controls. 
\end{remark}

We now recall the definition of a key class of controls, 
namely Markov policies.
A {\em Markov policy}
$\pi$ is a function 
$\pi:\; S\rightarrow A$
such that for each $x\in D$ there exists an adapted process $X$ 
on a stochastic basis satisfying
\begin{enumerate}[label=(\roman*)]
\item $X_0=x$;
\item $\Pi=\pi(X)$, defined by $\Pi_t:=\pi(X_t)$ for $t\geq0$, is in $\A_x$;
\label{enum:item_2_unique_in_law}
\item the processes $(X^\Pi,\Pi)$ and $(X, \pi(X))$ have the same law.
\label{enum:item_3_unique_in_law}
\end{enumerate}
Hereafter we denote such an $X$ by $X^\pi$.
Note that~\ref{enum:item_2_unique_in_law}
in the definition of a Markov policy
implies the existence of the process 
$X^\Pi$
and the uniqueness of the law of 
$(X^\Pi,\Pi)$.
Part~\ref{enum:item_3_unique_in_law} 
stipulates that the law of $(X,\pi(X))$
coincides with it.

\begin{remark}
As mentioned in Remark~\ref{rem:controls} above, our formulation of the control problem in~\eqref{eq:Value_Function}
does not make a reference to a particular filtered probability space. This 
allows us to consider the Markov control $\pi=\sgn$, see e.g. examples~\ref{item:Tanaka_control} and~\ref{item:Tanaka_control2} in 
Section~\ref{subsubsec:Weak} below. 
It is well known that the SDE
in~\eqref{eq:Tanaka_Control} (with $a=\sgn(X)$), arising in these examples, does not possess a strong solution,
and hence a strong formulation of the stochastic control problem would have to exclude such natural 
Markov controls. Furthermore, these examples show that such controls arise as the optimal controls in certain problems. 
\end{remark}

Given $x\in S$ and a policy $\Pi\in\A_x$ 
(resp. a Markov policy $\pi$),
we define the \textit{payoff} to be
\begin{equation}
\label{eq:VPi}
V^\Pi(x):=\E[J(x,\Pi)] \qquad\text{(resp. $V^\pi(x):=\E[J(x,\pi(X^\pi)]$)}.
\end{equation}
Hence the value function $V$, defined in~\eqref{eq:Value_Function}, can be expressed 
in terms of the payoffs 
$V^\Pi$
as 
\begin{equation}
\label{eq:Vx}
V(x):=\sup_{\Pi\in \A_x}V^\Pi(x)\qquad \text{for any $x\in S$.}
\end{equation}

\subsection{Examples}
There are numerous specific stochastic control problems 
that lie within the setting described
in Section~\ref{subsec:Setting}. We mention two
classes of examples. 
\subsubsection{Discounted infinite horizon problem.}
\label{subsubsec:Inf_hor}
Let $X^a$ be a killed Markov process with
$S=D\cup \{\partial\}$ with $\partial$ an
isolated cemetery state. Killing to $\partial$
occurs at a (possibly state and control-dependent) rate $\alpha$ 
and $\tau$ is the death time of the process.  
A special case of the controlled (killed) It\^o diffusion process
will be described in Section~\ref{subsec:discounted_controlled_diffusion}.
The detailed proofs that the Policy Improvement Algorithm from Section~\ref{sec:PIA} below 
can be applied in this case are given in~\cite{JMS_1}.
\begin{remark}
The  general setting allows us to consider more general problems where $\tau$ is the earlier of the killing time and exit from a domain. 
As is usual, we may also assume that the killing time is unobserved so that,
conditioning on the sample path and control we revise
problem~\eqref{eq:Value_Function} to the standard killed version, where 
$V^\Pi(x)$ and $V(x)$ are given in~\eqref{eq:VPi} and~\eqref{eq:Vx}, respectively, with 
$$J(x,\Pi):=\int_0^\tau
\exp\left(-\int_0^t\alpha(X^\Pi_s,
\Pi_s)ds\right)f(X^\Pi_t,
\Pi_t)dt+\exp\left(-\int_0^\tau\alpha(X^\Pi_s,
\Pi_s)ds\right)g(X^\Pi_\tau)1_{(\tau<\infty)}.
$$
\end{remark}
\subsubsection{The finite horizon problem.}
\label{subsubsec:Finf_hor}
Let $Y^a$ be a Markov process on a topological
space $S'$ with
infinitesimal generator $\G^a$ and $\tau$ the time
to the horizon $T$. Define $S:=S'\times \R$ and
$D:=S'\times \R_+$, so if $x=(y,T)$ then
$X^a_t=(Y^a_t,T-t)$, $\tau=T$ and
$\calL^a=\G^a-\frac{\partial}{\partial t}$. 
The detiled proofs that the PIA in Section~\ref{sec:PIA} below 
works in this setting are given in~\cite{JMS_2}.

\subsubsection{The weak formulation of the control problem is essential}
\label{subsubsec:Weak}
In this example we demonstrate that it is necessary to formulate 
the stochastic control setting in Section~\ref{subsec:Setting}
using the weak formulation in order not to exclude natural examples of the 
control problems. 

\begin{enumerate}[label=(\Roman*)]
\item In our formulation it is possible for two controls $\Pi$
\label{item:Tanaka_control}
and $\Sigma$
to have the same law but the pairs $(X^\Pi,\Pi)$
and
$(X^\Sigma,\Sigma)$ not to. Consider 
$S:=\R$,
$A:=\{-1,1\}$
and, for $a\in A$, 
the strong Markov process
$X^a$
is given by
\begin{equation}
\label{eq:Tanaka_Control}
d X^a_t = a\> dV_t,
\end{equation}
where 
$V$
is any Brownian motion. 
Let 
$W$
be a fixed Brownian motion on a stochastic basis.
Define 
$\Pi:=\sgn(W)$
(with 
$\sgn(0):=1$)
and 
in~\eqref{eq:Tanaka_Control} 
let $V$ be defined by the stochastic integral
$V_t:=\int_0^t\sgn(W_s)dW_s$.
Then $X^\Pi=W$
and
hence
$(X^\Pi,\Pi)=(W,\sgn(W))$.
Take
$\Sigma:=\sgn(W)$
and in~\eqref{eq:Tanaka_Control}
let
$V:=W$.
Then, by the Tanaka formula, we have
$$X^\Sigma_t =\int_0^t \sgn(W_s)\>dW_s= |W_t|-L_t^0(W),$$ 
where $L^0(W)$
is the local time of $W$ at zero. 
It is clear that 
$X^\Sigma$
is a Brownian motion
and hence 
$X^\Pi\ed X^\Sigma$
and
$\Pi\ed \Sigma$.
However the random vectors
$(X^\Pi,\Pi)$
and
$(X^\Sigma,\Sigma)$
have distinct joint laws,
e.g. 
$\P(X^\Pi_t>0,\Pi_t=-1)=0<\P(X^\Sigma_t>0,\Sigma_t=-1)$
for any 
$t>0$.

In order to show that such strategies can arise as optimal strategies,
consider (in the context of Section~\ref{subsubsec:Inf_hor})
the controlled process in~\eqref{eq:Tanaka_Control}
with 
$D:=(-1,1)$
and
$$
J(x,\Pi):=\exp\left(-\int_0^\tau \alpha(X^\Pi_t,\Pi_t)dt\right)\cdot g(X^\Pi_\tau),\quad\text{where} \quad
\alpha(x,a) := \begin{cases}
2+a, & x\in D, \\
\infty, & x\notin D,
\end{cases}$$
$\tau$ is the first exit of $X^\Pi$ from the interval $(-1,1)$ and 
$g:\{-1,1\}\to\R$ is given by
$$
g(1):=-\sinh(\sqrt{6}),\qquad g(-1):=\sqrt{3}\sinh(\sqrt{2}).
$$
Define the function 
$\widehat V:S\to\R$
by
$$
\widehat V(x) := \begin{cases}
-\sinh(\sqrt{6}x), & x\geq0 , \\
-\sqrt{3}\sinh(\sqrt{2}x),
 & x<0,
\end{cases}
$$
and note that 
$\widehat V$
is 
$C^1$,
piecewise
$C^2$
and the following equalities hold 
for all $x\in D\setminus\{0\}$:
$$\sgn(\widehat V(x))=-\sgn(x)
\quad\text{and}\quad
\widehat V''(x) = (2+4 \cdot 1_{\{x>0\}})\>\widehat V(x).
$$
Hence the following HJB equation holds (recall $a\in A=\{-1,1\}$ and thus $a^2=1$):
$$
\sup_{a\in A} \left[\frac{a^2}{2}\widehat V''-(a+2)\widehat V\right]=0,
\qquad
\text{with boundary condition $\widehat V|_{\partial D}=g$,}
$$
and the supremum is attained at 
$a=\sgn(x)$. 
Now, a standard application of martingale theory and stochastic calculus implies that the Markov policy
$\pi(x):=\sgn(x)$
is optimal 
for problem~\eqref{eq:Value_Function}
(with the controlled process given in~\eqref{eq:Tanaka_Control})
and its payoff 
$V^\pi$
equals 
the value function 
$V$
in~\eqref{eq:Value_Function}:
$
V(x)=V^\pi(x)=\widehat V(x)
$
for all $x\in D$.

\item It may appear at first glance that 
the weak formulation of the solution only played a role in Example~\ref{item:Tanaka_control}
due to the fact that the space of controls 
in~\ref{item:Tanaka_control}
was restricted to
$A=\{-1,1\}$.
Indeed, it holds that if in Example~\ref{item:Tanaka_control} we allow controls in the interval
$[-1,1]$,
then the Markov control 
$\pi(x)=\sgn(x)$
is no longer optimal
(as the HJB equation is no longer satisfied). 
However, 
the weak formulation of the control problem is essential even if we allow 
the controller to choose from an
uncountable set of actions at each moment in time.
We now illustrate this point by describe an example where the Markov control 
$\pi(x)=\sgn(x)$
is optimal, while the controls take values in the closed interval.

Consider 
the controlled process 
$X^a$
in~\eqref{eq:Tanaka_Control}
with 
$S$
and
$D$
as in 
Example~\ref{item:Tanaka_control}.
Let
$A:=[-1,1]$
and define 
\begin{equation}
\label{eq:J_ex_weak}
J(x,\Pi):=\exp\left(-\int_0^\tau \alpha(X^\Pi_t,\Pi_t)dt\right) \cdot g(X^\Pi_\tau)
-\int_0^\tau \exp\left(-\int_0^t \alpha(X^\Pi_s,\Pi_s)ds\right) \cdot f(X^\Pi_t) dt,
\end{equation}
where
$\tau$ is the first exit of $X^\Pi$ from the interval $D=(-1,1)$, 
$$
\alpha(x,a) := \begin{cases}
4a+9/2, & x\in D,\>a\in A, \\
\infty, & x\notin D,\>a\in A,
\end{cases},\qquad
f(x):=\frac{13}{2}\sinh(2\max\{x,0\}),\quad \text{for $x\in\R$,}
$$
and $g:\{-1,1\}\to\R$ is given by
$$
g(1):=-\sinh(2 ),\qquad g(-1):=2\sinh(1).
$$
Define the function 
$\widehat V:S\to\R$
by
$$
\widehat V(x) := \begin{cases}
-\sinh(2x), & x\geq0 , \\
-2\sinh(x),
 & x<0.
\end{cases}
$$
Note that 
$\widehat V$
is 
$C^1$,
piecewise
$C^2$
and, 
for all $x\in D\setminus\{0\}$,
it holds
$$\sgn(\widehat V(x))=-\sgn(x)
\quad\text{and}\quad
\widehat V''(x) = (1+3 \cdot 1_{\{x>0\}})\>\widehat V(x).
$$
We now show that the HJB equation 
\begin{equation}
\label{eq:HJB_weak_sol}
\sup_{a\in [-1,1]} \left[\frac{a^2}{2}\widehat V''-(4a+9/2)\widehat V\right]-f=0
\end{equation}
holds with boundary condition $\widehat V|_{\partial D}=g$
and the supremum attained at 
$a=\sgn(x)$. 
We first establish~\eqref{eq:HJB_weak_sol} for $x>0$. 
In this case~\eqref{eq:HJB_weak_sol} 
reads 
$$ \sup_{a\in [-1,1]} \left[\sinh(2x)\left((4a+9/2)-2a^2\right)\right]- \frac{13}{2}\sinh(2x)=0.$$
Now, the function $a\mapsto 4a-2a^2$ is increasing on $[-1,1]$. Hence the supremum is attained at $a=1$
and the equality follows. In the case
$x<0$,
the HJB equation in~\eqref{eq:HJB_weak_sol} takes the form
$$ \sup_{a\in [-1,1]} \left[-\sinh(x)\left(a^2/2-(4a+9/2)\right)\right]=0.$$
The function 
$a\mapsto a^2/2-4a-9/2$
is decreasing on the interval 
$[-1,1]$
and has a zero at
$a=-1$.
Hence the HJB equation in~\eqref{eq:HJB_weak_sol} holds with the stated boundary condition.
The classical martingale argument  
implies
that the Markov policy 
$\pi(x):=\sgn(x)$
is optimal for problem~\eqref{eq:Value_Function}
(with 
$J(x,\Pi)$
given in~\eqref{eq:J_ex_weak}) and its payoff 
$V^\pi$
equals the value function 
$V$
in~\eqref{eq:Value_Function}:
$
V(x)=V^\pi(x)=\widehat V(x)
$
for all $x\in D$.

\label{item:Tanaka_control2}
\end{enumerate}

\section{The policy improvement algorithm (PIA)}
\label{sec:PIA}
In order to develop the policy improvement algorithm, 
we first have to define the notion of an improvable Markov policy.

\begin{defin}
\label{def:Improvement}
A Markov policy $\pi$ is \textit{improvable} if 
$V^\pi\in \C$. The collection of improvable Markov policies is denoted by $I$.
A Markov policy $\pi'$ is an \textit{improvement} of $\pi\in I$ if, 
\begin{enumerate}[label=(\Roman*)]
\item for each $x\in D$
$$\pi'(x)\in \argmax_{a\in A}[\calL^aV^\pi(x)+f(x,a)],$$
or equivalently put
$$ \calL^{\pi'(x)}V^\pi(x)+f(x,\pi'(x))=\sup_{a\in A}[\calL^aV^\pi(x)+f(x,a)], $$
and
\label{enum:item_1_imporovable}
\item $\pi'$ is also a Markov policy.
\end{enumerate}
\end{defin}

\subsection{Improvement works}
\label{subsec:Improvment_works}
The PIA works by defining a sequence of improvements and their associated payoffs. More specifically,
$\pi_{n+1}$ is the improvement of the improvable Markov policy $\pi_n$ (in the sense of Definition~\ref{def:Improvement}). 
With this in mind, we make the following assumptions:

\begin{description}
\item[\namedlabel{As:A1}{\textbf{(As1)}}] 
there exists a non-empty subset $I^*$ of $I$  such that  $\pi_0\in I^*$ implies that, for each $n\in\N$, 
the Markov policy $\pi_n$ is a continuous function in $I^*$; 
\item[\namedlabel{As:A2}{\textbf{(As2)}}] 
for  any Markov policy $\pi_0\in I^*$, let the difference of consecutive payoff
processes converge in $L^1$ to a non-negative random variable:
$$\lim_{t\uparrow\infty}\left(V^{\pi_{n+1}}(X^{\pi_{n+1}}_{t\wedge\tau})-V^{\pi_{n}}(X^{\pi_{n+1}}_{t\wedge\tau})\right)\overset{L^1}{=}
Z_x\geq 0 \text{ a.s. for each $x\in D$.}
$$
\end{description}

\begin{remark}
The key assertions in Assumption~\ref{As:A1} are that,  
for every 
$n\in\N$,
the payoff 
$V^{\pi_n}$ is in $\C$ 
(see~\ref{item_2:Setting} in Section~\ref{subsec:Setting} for the definition
of $\C$) 
and that the $\sup$ in~\ref{enum:item_1_imporovable} of Definition~\ref{def:Improvement}
is attained.
\end{remark}

The following theorem asserts that the algorithm, under the assumptions above, actually improves improvable policies. 
We prove it in Section~\ref{subsec:proof_thm1} below.

\begin{thm}
\label{thm:1}
Under Assumptions~\ref{As:A1} and~\ref{As:A2}, the inequality
$$V^{\pi_{n+1}}(x)\geq V^{\pi_{n}}(x)\qquad\text{holds for each $n\in\N$ and all $x\in S$.}  $$
\end{thm}

\subsection{Convergence of payoffs}
Assume from now on that Assumptions~\ref{As:A1} and~\ref{As:A2} hold and that we have fixed an improvable Markov policy 
$\pi_0$ in $I^*$. Denote by
$(\pi_n)_{n\in\N}$ 
the sequence of Markov policies in $I^*$
defined by the PIA 
started at
$\pi_0$
(see the beginning of Section~\ref{subsec:Improvment_works}).

\begin{description}
\item[\namedlabel{As:A3}{\textbf{(As3)}}] The value function $V$, defined in~\eqref{eq:Value_Function}, is finite on the domain $D$.
\item[\namedlabel{As:A4}{\textbf{(As4)}}] There is a subsequence $(n_k)_{k\in\N}$ such that
$$
\lim_{k\nearrow\infty}\calL^{\pi_{n_k+1}}V^{\pi_{n_k}}(x)+f(x,\pi_{n_k+1}(x))= 0\quad\text{ uniformly in $x\in D$.}
$$
\item[\namedlabel{As:A5}{\textbf{(As5)}}]  For each $x\in S$, each $\Pi\in \A_x$ and each $n\in\N$
the following limit holds:
$$
V^{\pi_n}(X^\Pi_{t\wedge \tau})\L g(X^\Pi_{\tau})1_{(\tau<\infty)}\qquad\text{as $t\to\infty$.}
$$
\end{description}

The next result states that the PIA works. Its proof is in Section~\ref{subsec:proof_thm2} below.

\begin{thm}
\label{thm:2}
Under Assumptions~\ref{As:A1}--\ref{As:A5}, the following limit holds:
$$
V^{\pi_n}(x)\uparrow V(x)\qquad\text{for all $x\in S$.}
$$
\end{thm}

\subsection{Convergence of policies}
Assume from now on that Assumptions~\ref{As:A1}--\ref{As:A5}
hold and that, as before, we have fixed a $\pi_0$ in $I^*$
together with 
the sequence of improved Markov policies 
$(\pi_n)_{n\in\N}$.

\begin{description}
\item[\namedlabel{As:A6}{\textbf{(As6)}}] For any $\pi_0\in I^*$, the sequence $(\pi_n)_{n\in\N}$ is sequentially precompact in the 
topology of uniform convergence on compacts on the space of continuous functions $C(S,A)$.
\item[\namedlabel{As:A7}{\textbf{(As7)}}] For any sequence $(\rho_n)_{n\in\N}$ in $I^*$, 
such that
\begin{enumerate}[label=(\roman*)]
\item $\exists$ Markov policy $\rho$, such that $\rho_n\nto \rho$ uniformly on compacts in $S$, 
\item $\phi_n\in \C$ for all $n\in\N$ and $\phi_n\nto\phi$ pointwise,
\item $\calL^{\rho_n}\phi_n\nto Q$ uniformly on compacts in $S$,
\end{enumerate}
then
$$
\phi\in \C,\quad \calL^{\rho}\phi=Q \text{ and } \calL^{\rho}\phi_n-\calL^{\rho_n}\phi_n\nto 0
$$
uniformly on compacts in
$S$.
\item[\namedlabel{As:A8}{\textbf{(As8)}}] For each $x\in D$ and each $\Pi\in \A_x$,
$$
V(X^\Pi_{t\wedge \tau})\L g(X^\Pi_{\tau})1_{(\tau<\infty)}\qquad\text{as $t\to\infty$,}
$$
holds.
\end{description}

The next theorem states that the sequence of policies produced by the PIA
contains a uniformly convergent subsequence.  
We give a proof of this fact in Section~\ref{subsec:proof_thm3} below.

\begin{thm}
\label{thm:3}
Under Assumptions~\ref{As:A1}--\ref{As:A8}, for any $\pi_0$ in $I^*$
and the corresponding sequence of improved Markov policies 
$(\pi_n)_{n\in\N}$,
there exists a subsequence $(\pi_{n_k})_{k\in\N}$ such that $\pi_{n_k}\nto \pi^*$ 
in the topology of uniform convergence on compacts 
and $V^{\pi^*}=V$.
\end{thm}

\section{Examples of the PIA}
\label{sec:Exmaples_diffusions}
\subsection{Discounted infinite horizon controlled diffusion.}
\label{subsec:discounted_controlled_diffusion}
This section gives an overview of the results in~\cite{JMS_1}.
Define $D:=\R^d$ and $S:=\R^d\cup\{\partial\}$ and let 
$\C=C^2_b(\R^d,\R)$ be the space of bounded, $C^2$, real-valued functions on $\R^d$. 
Suppose that $X$ is a controlled (killed) It\^o diffusion in $\R^d$, 
so that 
\begin{equation}
\label{eq:diff_gen}
\calL^a\phi(\cdot)=\half\sigma(\cdot,a)^TH\phi \sigma(\cdot,a)+\mu(\cdot,a)^T\nabla \phi-\alpha(\cdot,a)\phi,
\end{equation}
where $H\phi$ (resp. $\nabla \phi$) denotes the Hessian 
(resp. gradient) with entries $\frac{\partial^2\phi}{\partial x_i\partial x_j}$, $1\leq i,j \leq d$
(resp. $\frac{\partial\phi}{\partial x_i}$, $1\leq i \leq d$). 
Furthermore, we make the following assumptions on the deterministic characteristics of the model:
\begin{description}
\item[\namedlabel{aleph1}{\textbf{(\boldmath$\aleph1$)}}] 
$\sigma(x,a)$,
$\mu(x,a)$, $\alpha(x,a)$ and $f(x,a)$ are uniformly (in $a$) Lipschitz on
compacts in $\R^d$ and are continuous in $a$; $\alpha$ is bounded below by
a positive constant $\lambda>0$, $\sigma$ is uniformly elliptic and $f$ is uniformly bounded by
a (large) constant $M$.  
\item[\namedlabel{aleph2}{\textbf{(\boldmath$\aleph2$)}}] The control set $A$ is a compact
interval $[a,b]$.
\end{description}

For every $h\in \C$ and $x\in \R^d$, let $I_h(x)$ denote an element of
$\argmax_{a\in A}[\calL^a h(x,a)+f(x,a)]$.

\begin{description}
\item[\namedlabel{aleph3}{\textbf{(\boldmath$\aleph3$)}}] 
If the sequence of functions $(h_n)_{n\in\N}$ is in $C^2$ and the sequence $(Hh_n)_{n\in\N}$ is uniformly bounded on compacts,
then we may choose the sequence of functions $(I_{h_n})_{n\in\N}$ to be uniformly Lipschitz on compacts.
\end{description}

\begin{remark}
\begin{enumerate}
\item The assumption in~\ref{aleph3} is very strong. Neverthless, if $\sigma$ is
independent of $a$ and bounded, $\mu(x,a)=\mu_1(x)-ma$, $\alpha(x,a)=\alpha_1(x)+ca$
and $f(x,a)=f_1(x)-f_2(a)$ with $f_2\in C^1$ and with strictly positive
derivative on $A$, and Assumptions~\ref{aleph1} and~\ref{aleph1} hold, then~\ref{aleph3}  
holds.
\item We stress that the assumptions in \ref{aleph1}--\ref{aleph3}
do not depend on the stochastic behaviour of the model but are given explicitly in terms
of its deterministic characteristics. This makes the PIA provably convergent for a broad
class of diffusion control problems. 
\end{enumerate}
\end{remark}

\begin{proposition}
\label{prop:4}
Under Assumptions~\ref{aleph1}--\ref{aleph3}, Assumptions~\ref{As:A1}--\ref{As:A8} 
hold for the (possibly killed) controlled diffusion process with generator~\eqref{eq:diff_gen} 
and the PIA converges when started at any locally Lipschitz Markov policy $\pi_0$.
\end{proposition}

\begin{proof}
Note that 
$\calL^a\phi$ 
is jointly continuous if 
$\phi$ is in $\C$ and (with the usual trick to deal with killing) \eqref{gen} 
holds for any control $\Pi$ such that there is
a solution to the killed equation 
$$
X^\Pi_t=(x+\int_0^t\sigma(X^\Pi_s,\Pi_s)dB_s+\int_0^t \mu(X^\Pi_s,\Pi_s)ds)1_{(t<\tau)}+\partial 1_{(t\geq \tau)}.
$$
Furthermore, any locally Lipschitz $\pi$ is a Markov policy by strong uniqueness of the solution to the SDE.  
We now establish Assumptions~\ref{As:A1}--\ref{As:A8}. 

\noindent \ref{As:A1} If $\pi_0$ is Lipschitz on compacts then by Assumption~\ref{aleph3}, \ref{As:A1} holds.\\
\noindent \ref{As:A3} Boundedness of $V$ in~\ref{As:A3} follows from the boundedness of $f$ and the fact that $\alpha$ is bounded away from 0.\\
\noindent \ref{As:A6} Assumption~\ref{aleph3}  implies that $(\pi_n)$ are uniformly
Lipschitz and hence sequentially precompact in the sup-norm topology (A6) by
the Arzela-Ascoli Theorem.\\  
\noindent \ref{As:A5}  $g=0$ and since $\alpha$ is bounded
away from 0, for any $\Pi$, $X^\Pi_t\rightarrow \partial$. Now
$V^{\pi_n}(\partial)=0$ and so, by bounded convergence, \ref{As:A5} holds:
$$
V^{\pi_n}(X^\Pi_{t\wedge \tau})\L g(X^\Pi_{\tau})1_{(\tau<\infty)}\qquad\text{as $t\to\infty$.}
$$
\noindent \ref{As:A2}   Similarly, \ref{As:A2} holds:
$$V^{\pi_{n+1}}(X^{n+1}_{t\wedge\tau})-V^{\pi_{n}}(X^{n+1}_{t\wedge\tau})\L 0\qquad\text{as $t\to\infty$.}
$$
\noindent \ref{As:A4} The statement in~\ref{As:A4} is trickier to establish. Note that we have~\ref{As:A1} and~\ref{As:A2}, by
Theorem~\ref{thm:1}, we know that $V^{\pi_n}(x)$ is a non-decreasing sequence. Moreover, since~\ref{As:A3}  holds,  $V^{\pi_n}\uparrow V^{lim}$.
Now take a subsequence $(n_k)_{k\in\N}$ such that $(\pi_{n_k},\pi_{n_k+1})\rightarrow (\pi^*,\tilde\pi)$ 
uniformly on compacts.  
Then the corresponding $\sigma$ etc.  must also converge. Denote the limits by $\sigma^*$, $\tilde \sigma$ etc.
Then, $\Vl\in \C^2_b$ (see the argument in~\cite{JMS_1}, based on coupling and the classical PDE theory from Friedman~\cite{Friedman}) and 
$$\lim_{k\to\infty }\nabla V^{\pi_{n_k}}=\lim_{k\to\infty }\nabla V^{\pi_{n_k+1}}= \nabla \Vl
\quad\text{ and }\quad \lim_{k\to\infty }HV^{\pi_{n_k}}=\lim_{k\to\infty }H V^{\pi_{n_k+1}}=H\Vl$$
uniformly on compacts and 
$\calL^{\tilde \pi}\Vl+f(\cdot, \tilde\pi(\cdot))=0$. 
Now, from the convergence of the
derivatives of 
$V^{\pi_{n_k}}$, we obtain 
$$\calL^{\pi_{n_k+1} }V^{\pi_{n_k}}+f(\cdot, \pi_{n_k+1}(\cdot))\rightarrow \calL^{\tilde \pi}\Vl+f(\cdot, \tilde\pi(\cdot))=0$$ 
uniformly on compacts.\\
\noindent \ref{As:A7} and~\ref{As:A8} follow from Friedman~\cite{Friedman}.
See~\cite{JMS_1} for details.
\end{proof}

\subsection{Finite horizon controlled diffusion.}
This is very similar to the previous example if we add the requirement that $g$ is Lipschitz and bounded.
The details can be found in~\cite{JMS_2}.

\begin{remark}
In both examples we need to prove that $V^{\pi_n}$ is continuous before we can apply the usual PDE arguments.
This crucial step is carried out in~\cite{JMS_1} and~\cite{JMS_2} respectively. 
\end{remark}

\section{Proofs}
\label{sec:Proofs}

\begin{lem}
\label{lem:lemma4}
Under Assumptions \ref{As:A1} and~\ref{As:A2}, it holds that  
$$\calL^{\pi_n}V^{\pi_n}(x)+f(x,\pi_n(x))=0 \qquad\text{ for all $x\in D$ and $n\in\N$.}
$$
\end{lem}
\begin{proof} We know that
$$
V^{\pi_n}(X^{\pi_n}_{t\wedge \tau})-\int_0^{t\wedge \tau} \calL^{\pi_n}V^{\pi_n}(X^{\pi_n}_s)ds
$$
is a martingale and the usual Markovian argument shows that therefore
$$
\int_0^{t\wedge \tau} (\calL^{\pi_n}V^{\pi_n}+f(\cdot,\pi_n(\cdot))(X^{\pi_n}_s)ds
=0.
$$
The result then follows from continuity of $\calL^{\pi_n}V^{\pi_n}+f(\cdot,\pi_n(\cdot))$ 
(see~\ref{item_2:Setting} in Section~\ref{subsec:Setting})
and the right continuity of $X^{\pi_n}$.
\end{proof}

\subsection{Proof of Theorem~\ref{thm:1}} 
\label{subsec:proof_thm1}
Take $\pi_0\in I^*$ and $x\in D$ 
and let
$(\pi_n)_{n\in\N}$
be the sequence of policies produced by the PIA.
For any
$n\in\N$ 
define
$$
S_t:=(V^{\pi_{n+1}}-V^{\pi_n})(X^{\pi_{n+1}}_{t\wedge \tau}),\qquad\text{$t\geq0$.}
$$
By assumption, both payoffs $V^{\pi_{n+1}}$ and $V^{\pi_n}$ are in $\C$. Hence the process
$$
V^{\pi_k}(X^{\pi_{n+1}}_{t\wedge \tau})-\int_0^{t\wedge \tau} \calL^{\pi_{n+1}}V^{\pi_k}(X^{\pi_{n+1}}_s)ds, \qquad\text{for $k=n$, $n+1$,}
$$
is a martingale. 
So,
$$
S_t=(V^{\pi_{n+1}}-V^{\pi_n})(x)+M_{t\wedge \tau}+\int_0^{t\wedge \tau} (\calL^{\pi_{n+1}}V^{\pi_{n+1}}-\calL^{\pi_{n+1}}V^{\pi_n})(X^{\pi_{n+1}}_s)ds,
$$
where $M$ is a martingale.
Thus
$$
S_t=(V^{\pi_{n+1}}-V^{\pi_n})(x)+M_{t\wedge \tau}-\int_0^{t\wedge \tau} \sup_{a\in A}[\calL^aV^{\pi_n}+f(\cdot,a)](X^{\pi_{n+1}}_s)ds,
$$
by Lemma~\ref{lem:lemma4} and the definition of $\pi_{n+1}$.
Appealing to Lemma~\ref{lem:lemma4} again, the integrand is non-negative and hence $S$ is a
supermartingale. Taking expectations and letting $t\rightarrow \infty$ we
obtain the result using~\ref{As:A2}.\hfill$\square$

\subsection{Proof of Theorem~\ref{thm:2}} 
\label{subsec:proof_thm2}
From Theorem~\ref{thm:1} and~\ref{As:A3}, $V^{\pi_n}(x)\uparrow \Vl(x)$ holds for any $x\in S$ as $n\to\infty$,
where the function $\Vl$ is finite and bounded above by $V$.
Fix $x\in D$ and $\Pi\in \A_x$ and take the subsequence 
$(\pi_{n_k})_{k\in \N}$
in~\ref{As:A4}.
Set
$$
S^k_t=
V^{\pi_{n_k}}(X^{\Pi}_{t\wedge \tau})+\int_0^{t\wedge \tau} f(X^{\Pi}_s, {\Pi}_s)ds.
$$
It follows that
there is a martingale $M^k$ such that
\begin{eqnarray*}
S^k_t&=& S^k_0+M^k_{t\wedge \tau}+\int_0^{t\wedge \tau}[\calL^{\Pi_s}V^{\pi_{n_k}}+f(\cdot,\Pi_s)](X^{\Pi}_s)ds\\
&\leq & 
S^k_0+M^k_{t\wedge \tau}+\int_0^{t\wedge \tau}\sup_{a\in A}[\calL^aV^{\pi_{n_k}}+f(\cdot,a)](X^{\Pi}_s)ds\\
&= &S^k_0+M^k_{t\wedge \tau}+\int_0^{t\wedge \tau}[\calL^{\pi_{n_k+1}}V^{\pi_{n_k}}+f(\cdot,{\pi_{n_k+1}}(\cdot))](X^{\Pi}_s)ds
\end{eqnarray*}
So
\begin{equation}
\label{eq:supermart}
\E S^k_t\leq 
V^{\pi_{n_k}}(x) +\E\int_0^{t\wedge \tau}[\calL^{\pi_{n_k+1}}V^{\pi_{n_k}}+f(\cdot,{\pi_{n_k+1}}(\cdot))](X^{\Pi}_s)ds.
\end{equation}

Letting $k\rightarrow \infty$ in~\eqref{eq:supermart} we obtain, by~\ref{As:A4}, together with dominated convergence,  
and monotone convergence, that
\begin{equation}
\label{eq:Lower_bound_V_lim}
\Vl(x)\geq \E[\Vl(X^{\Pi}_{t\wedge \tau})+\int_0^{t\wedge \tau} f(X^{\Pi}_s, \Pi_s)ds]
\qquad\text{for all $t\geq0$.}
\end{equation}
Now~\ref{As:A5} 
and Fatou's lemma
(recall $\Vl\geq V^{\pi_{n_k}}\geq0$ for any index $k$) 
imply
$$
\liminf_{t\to\infty}\E\Vl(X^{\Pi}_{t\wedge \tau})
\geq 
\liminf_{t\to\infty}\E V^{\pi_{n_k}}(X^{\Pi}_{t\wedge \tau})
\geq 
\E\liminf_{t\to\infty} V^{\pi_{n_k}}(X^{\Pi}_{t\wedge \tau})
=
\E g(X^\Pi_{\tau})1_{(\tau<\infty)},
$$
and so~\eqref{eq:Lower_bound_V_lim} yields $\Vl(x)\geq V^\Pi(x)$ for each $\Pi\in\A_x$. Hence $\Vl\geq V$ on the domain 
$D$ (on the complement we clearly have $\Vl= V$). However, 
since by definition 
$V^{\pi_n}\leq \Vl$ 
on
$S$
for all $n\in\N$,
$\Vl=\lim_{n\to\infty}V^{\pi_n}\leq V$ on $S$ so in fact we have equality.\hfill$\square$

\subsection{Proof of Theorem~\ref{thm:3}} 
\label{subsec:proof_thm3}
Let 
$(\pi_{n_j})_{j\in\N}$ 
be a subsequence, guaranteed by~\ref{As:A6}, of the sequence of Markov policies 
$(\pi_{n})_{n\in\N}$ 
in
$I^*$
produced by the PIA.
Put differently, the limit 
$$\lim_{j\to\infty}\pi_{n_j}= \pi^*$$
holds uniformly on compacts in $S$
for some 
Markov policy 
$\pi^*$.
Hence, for any 
$x\in D$, 
there exists (by the definition of a Markov policy) a controlled process
$X^{\pi^*}$
defined on some filtered probability space. 

Fix $x\in D$, define the process $S^j=(S^j_t)_{t\geq0}$,
\begin{equation}
\label{eq:def_Sj}
S^j_t:=
V^{\pi_{n_j}}(X^{\pi^*}_{t\wedge \tau})+\int_0^{t\wedge \tau} f(X^{\pi^*}_s, {\pi^*}(X^{\pi^*}_s))ds,
\end{equation}
and note that the following equality holds 
\begin{equation}\label{pol}
S^j_t=V^{\pi_{n_j}}(x)+M_t +\int_0^{t\wedge \tau} \left[\calL^{\pi^*}V^{\pi_{n_j}}+f(\cdot,\pi^*(\cdot))\right](X^{\pi^*}_s)ds
\qquad \text{for any $t\geq0$,}
\end{equation}
where the martingale $M=(M_t)_{t\geq0}$
is given by
$$ M_t:=
V^{\pi_{n_j}}(X^{\pi^*}_{t\wedge \tau})-V^{\pi_{n_j}}(x)-\int_0^{t\wedge \tau} \calL^{\pi^*}V^{\pi_{n_j}}(X^{\pi^*}_s))ds.
$$
By Lemma~\ref{lem:lemma4} and the representation in~\eqref{pol}
we obtain
$$
S^j_t=V^{\pi_{n_j}}(x)+M_t +\int_0^{t\wedge \tau} [(\calL^{\pi^*}-\calL^{\pi_{n_j}})V^{\pi_{n_j}}+f(\cdot,\pi^*(\cdot))-f(\cdot,\pi_{n_j}(\cdot))](X^{\pi^*}_s)ds.
$$
Take expectations on both sides of this identity. 
By localising, applying~\ref{As:A7} and using Theorem~\ref{thm:2} we obtain
\begin{equation}\label{first_lim}
V(x) = \lim_{j\to\infty} \E[S_t^j]\qquad \text{for any $t\geq0$.}
\end{equation}

Definition~\eqref{eq:def_Sj} and Theorem~\ref{thm:2} imply a.s. monotone convergence 
$$S^j_t\nearrow 
V(X^{\pi^*}_{t\wedge \tau})+\int_0^{t\wedge \tau} f(X^{\pi^*}_s, {\pi^*}(X^{\pi^*}_s))ds
\qquad\text{as $j\to\infty$ for any $t\geq0$.}
$$
Hence by the monotone convergence theorem and~\eqref{first_lim} we find
$$
V(x) = \E[\lim_{j\to\infty} S_t^j] = \E\left[V(X^{\pi^*}_{t\wedge \tau})+\int_0^{t\wedge \tau} f(X^{\pi^*}_s, {\pi^*}(X^{\pi^*}_s))ds\right]
\qquad \text{for any $t\geq0$.}
$$
Letting $t\rightarrow \infty$, applying~\ref{As:A8} and recalling the definition 
of $V^{\pi^*}$
yields the result that $V^{\pi^*}=V$. \hfill $\square$

\end{document}